\documentclass[leqno,12pt]{amsart}
\usepackage{setspace}
\usepackage[english]{babel}
\usepackage{bussproofs}
\usepackage{csquotes}
\usepackage{dirtytalk}
\usepackage{mathtools}
\usepackage{mathrsfs}
\usepackage{latexsym}
\usepackage{enumitem}
\usepackage{amsthm}
\usepackage{amssymb}
\DeclareMathAlphabet{\mathpzc}{OT1}{pzc}{m}{it}
\usepackage[latin1]{inputenc}
\usepackage{afterpage}

\usepackage{bbm}
\usepackage[rgb,dvipsnames]{xcolor}
\usepackage{tikz} % for diagrams and drawing
\usetikzlibrary{decorations.text}
\usetikzlibrary{arrows,arrows.meta,hobby}
\usetikzlibrary{shapes.misc, positioning}

\usepackage{tikz-cd}

\usetikzlibrary{cd}

\newfont{\ssi}{cmssi12 at 12pt}

\newenvironment{ea*}{\begin{eqnarray*}}{\end{eqnarray*}}
\setbox0=\hbox{$\longrightarrow$}

\newcommand{\bG}{{\bar{G}}}

\newtheorem{theorem}{Theorem}[section]

\newtheorem{lemma}[theorem]{Lemma}

\newtheorem{corollary}[theorem]{Corollary}
\newtheorem{observation}[theorem]{Observation}

\newtheorem{claim}[theorem]{Claim}
\theoremstyle{definition}
\newtheorem{definition}[theorem]{Definition}

\theoremstyle{remark}

\newtheorem{question}{Question}

\def\hook{\upharpoonright}
\def\forces{\Vdash}

\def\ZFC{\mathsf{ZFC}}

\def\PFA{\mathsf{PFA}}
\def\MA{\mathsf{MA}}

\def\DCFA{\mathsf{DCFA}}

\def\mfc{\mathfrak{c}}

\def \mfd{\mathfrak{d}}

\def\CH {\mathsf{CH}}
\def\P{\mathbb P}
\def\:{\mathrel{\lower.04em\hbox{\rlap{$\cdot$}}{:}}}

\usepackage[margin=1.3in]{geometry}

\begin{document}

\title{Specializing Wide Aronszajn Trees Without Adding Reals}

\author[Switzer]{Corey Bacal Switzer \\ The Graduate Center of the City University of New York}
\address[C.~B.~Switzer]{Mathematics, The Graduate Center of The City University of New York, 365 Fifth Avenue, New York, NY 10016}
\email{cswitzer@gradcenter.cuny.edu}
\urladdr{https://coreyswitzer.wordpress.com/}
%\thanks{This reasearch was supported by a CUNY mathematics fellowship and the author would like to thank the mathematics department at the Graduate Center at CUNY for this. The author would also like to thank his advisor Joel David Hamkins for listening patiently and giving thoughtful advice and enthusiastic encouragement in the early stages of this work.}
%\subjclass[2010]{03E17, 03E05, 03E15, 03E35}

\date{}

\maketitle

\begin{abstract}
We show that under certain circumstances wide Aronszajn trees can be specialized iteratively without adding reals. We then use this fact to study forcing axioms compatible with $\CH$ and list some open problems.
\end{abstract}

\section{Introduction}

The purpose of this note is to prove a technical strengthening of a theorem of Shelah's on specializing Aronszajn trees and connect it to some open problems in iterated forcing and the continuum hypothesis. Specializing Aronszajn trees iteratively without adding reals goes back to the work of Jensen separating $\CH$ from $\diamondsuit$, see \cite{DevJohSH}. Reworking this result, in \cite[Chapter V]{PIPShelah} Shelah introduced the class of dee-complete and ${<}\omega_1$-proper forcing notions, an iterable class which does not add reals and showed that there is a forcing notion in this class that specializes Aronszajn trees. However the countability of the levels of the trees is essential in Shelah's proof in contrast to the ccc specializing forcing introduced in \cite{BMR70}, which adds reals but where the width of the tree plays no role. Therefore it remains unclear when one can specialize wider trees without adding reals. In this note we provide a partial solution to this problem by proving that under certain circumstances there are dee-complete and ${<}\omega_1$-proper posets to specialize wide trees. Specifically we show the following.

\begin{theorem}
Suppose $T$ is an $\omega_1$-tree (countable levels, but potentially uncountable branches) and $S \subseteq T$ is a wide Aronszajn tree with the induced suborder. Then there is a forcing notion $\mathbb P = \mathbb P_{S, T}$ which specializes $S$ and is dee-complete and ${<}\omega_1$-proper.
\label{specwidetree}
\end{theorem}

Using this poset we give an application to forcing axioms compatible with $\CH$.
\begin{theorem}
Under the forcing axiom for dee-complete and ${<}\omega_1$-proper forcing notions, all $\omega_1$-trees are essentially special and therefore there are no Kurepa trees.
\end{theorem}
This latter theorem was shown by Shelah under the additional assumption that $\CH$ and $2^{\aleph_1} = \aleph_2$ holds. What's new here is that using the forcing notion from Theorem \ref{specwidetree} we can remove the cardinal arithmetic assumption. 

The general question of when one can specialize a wide tree without adding reals turns out to be very interesting and there are many open questions still. The note finishes with some brief further observations and open problems. In particular, a connection to cardinal characteristics is observed.

\section{Preliminaries: Dee-complete Forcing, ${<}\omega_1$-Properness and Trees}

\subsection{Strengthening Properness}

 Given a model $N$ which elementarily embeds in some $H_\theta$, a forcing notion $\mathbb P \in N$ and a condition $p \in \mathbb P$ write ${\rm Gen}(N, \mathbb P, p)$ for the set of $\mathbb P$-generic filters over $N$ containing $p$. The following definitions come from \cite[Chapter V]{PIPShelah} and a particularly good exposition is also given in \cite{AbrahamHB}. What I call a completeness system here is called a \say{countably complete} completeness system in \cite{PIPShelah}. However, every completeness system considered in this note is countably complete so I omit the additional notation.

\begin{definition}
A {\em completeness system} is a function $\mathbb D$ defined on some set of triples $(N, \mathbb P, p)$ such that $N \prec H_\theta$ for some $\theta$, $\mathbb P \in N$ is a forcing notion and $p \in \mathbb P \cap N$ is a condition and the following hold:
\begin{enumerate}
\item
$\mathbb D(N, \mathbb P, p)$ is a family of sets, $A$, such that each $A \subseteq {\rm Gen}(N, \mathbb P, p)$.
\item
If $A_i \in \mathbb D(N, \mathbb P, p)$ for each $i < \omega$ then the intersection $\bigcap_{i < \omega} A_i$ is non-empty.
\end{enumerate}
If for a fixed $\mathbb P$ and some cardinal $\theta$, if $\mathbb D$ is defined on the set of all triples $(N, \mathbb P, p)$ with $p \in \mathbb P \in N$, $p \in N$ and $N \prec H_\theta$ then we call $\mathbb D$ a completeness system on $\theta$ for $\P$.
\end{definition}
Completeness systems in general are quite easy to construct, which leads one to question their utility. In general we will only be interested therefore in ones which are \say{nicely defined}, a notion Shelah refers to as {\em simple}.
\begin{definition}
A completeness system $\mathbb D$ is {\em simple} if there is a formula $\phi$ and a parameter $s \in H_{\omega_1}$ such that $\mathbb D(N, \mathbb P, p) = \{A^{N, \P, p}_u \; | \; u \in H_{\omega_1}\}$ where $A^{N, \P, p}_u$ is defined as follows: for $N\prec H_\theta$, let $\bar{N}$ be the Mostowski collapse of $N$ and $\pi_N:\bar{N} \to N$ the inverse of the Mostowski collapse. We let $\overline{A}^{N, \mathbb P, p}_u := \{\bG \in {\rm Gen}(\bar{N}, \pi^{-1}(\P), \pi_N^{-1}(p)) \; | \; H_{\omega_1} \models \phi (\bar{N}, \bG, \pi^{-1}(\P), \pi_N^{-1}(p), u, s)\}$. Finally let $A^{N, \P, p}_u$ be the set of generics generated by $\pi_N``\bG$ for $\bG \in \overline{A}^{N, \mathbb P, p}_u$.
\end{definition}
Using this, I can define dee-completeness.
\begin{definition}
We say that $\mathbb P$ is {\em dee-complete} if for every sufficiently large $\theta$ there is a simple completeness system $\mathbb D$ on $\theta$ for $\P$ such that whenever $\mathbb P \in N \prec H_\theta$, with $N$ countable and $p \in \mathbb P \cap N$ there is an $A \in \mathbb D(N, \mathbb P, p)$ such that for all $\overline{G} \in A$ there is a condition $q \in \mathbb P$ so that $q \leq r$ for all $r \in \overline{G}$.
\end{definition}
Given a poset $\mathbb P$ we say that a (not necessarily simple) completeness system $\mathbb D$ is a {\em completeness system for} $\mathbb P$ if it satisfies the requirements of the definition of dee-completeness. Observe that the existence of a completeness system for $\mathbb P$ implies that $\mathbb P$ is proper and adds no new reals (or indeed $\omega$ sequences of elements from $V$) since the condition $q$ as in the definition of dee-completeness is an $(N, \mathbb P)$-generic condition and if $\dot{a}: \check{\omega} \to \check{V}$ names an $\omega$-sequence, then there is a model $N \ni \dot{a}$ and a $\P$-generic $G$ over $N$ which has a lower bound $q$ so $q$ decides $\dot{a}(\check{n})$ for all $n < \omega$.

\begin{definition}[$\alpha$-Properness]
Let $\theta$ be a cardinal and $\alpha < \omega_1$. An $\alpha$-{\em tower} for $H_\theta$ is a sequence $\vec{N} = \langle N_i \; | \; i < \alpha \rangle$ of countable elementary substructures of $H_\theta$ so that for each $\beta < \alpha$, we have $\langle N_i \; | \; i \leq \beta \rangle \in N_{\beta + 1}$ and if $\lambda < \alpha$ is a limit ordinal then $N_\lambda = \bigcup_{i < \lambda} N_i$. We say that $\mathbb P$ is $\alpha$-proper if for all sufficiently large $\theta$, all $p \in \mathbb P$ and all $\alpha$-towers $\vec{N}$ in $H_\theta$ so that $p, \mathbb P \in N_0$ there is a $q \leq p$ which is simultaneously $(N_i, \mathbb P)$-generic for every $i < \alpha$. We say that $\mathbb P$ is ${<}\omega_1$-proper if it is $\alpha$-proper for all $\alpha < \omega_1$.
\end{definition}
Note that properness is $1$-properness. The point is the following iteration theorem due to Shelah, \cite[Chapter V, Theorem 7.1]{PIPShelah}.
\begin{theorem}
If $\langle (\mathbb P_\alpha, \dot{\mathbb Q}_\alpha) \; | \; \alpha < \nu\rangle$ is a countable support iteration of some length $\nu$ so that for each $\alpha < \nu$, $\forces_{\mathbb P_\alpha} ``\dot{\mathbb Q}_\alpha$ is dee-complete and ${<}\omega_1$-proper", then $\mathbb P_\nu$ is dee-complete and ${<}\omega_1$-proper. In particular such iterations do not add reals.
\end{theorem}

As an immediate consequence, we obtain, relative to a supercompact, the consistency of $\DCFA$, the forcing axiom for dee-complete and ${<}\omega_1$-proper forcing notions and even its consistency with $\CH$. Of course $\DCFA$ does not imply $\CH$ since $\PFA$ implies $\DCFA$ trivially. Very little attention has gone into $\DCFA$ as an axiom in its own right outside of \cite{JensenCH}. However one notable exception is \cite{AbrahamTodPP} where it is shown that $\DCFA + \CH$ implies the P-Ideal Dichotomy.

\subsection{Trees}
The main purpose of this note is to look at applications of dee-complete forcing to trees. Let me review some notation and terminology related to this here for reference. Recall that a {\em tree} $T = \langle T, \leq_T\rangle$ is a partially ordered set so that for each $t \in T$ the set of $s \leq_T t$ is well ordered. A {\em branch} through a tree is a maximal linearly ordered subset.

\begin{definition}
Let $T$ be a tree, $\alpha$ an ordinal and $\kappa$ and $\lambda$ cardinals.
\begin{enumerate}
\item
The $\alpha^{\rm th}$-{\rm level} of $T$, denoted $T_\alpha$ is the set of all $t \in T$ so that $\{s \; | \; s <_T t\}$ has order type $\alpha$. Also let $T_{\leq \alpha} = \bigcup_{i \leq \alpha} T_i$ and $T_{< \alpha} = \bigcup_{i < \alpha} T_i$.
\item
The {\em height} of $T$ is the least $\alpha$ with $T_\alpha = \emptyset$. 
\item
If $\alpha < \beta$ are ordinals, $T$ is a tree of height at least $\beta + 1$ and $t \in T_\beta$ then denote by $t \hook \alpha$ the unique $s \in T_\alpha$ so that $s \leq_T t$.
\item
We say that $T$ is a $\kappa$-tree if it has height $\kappa$ and each level has size $<\kappa$.
\item
$T$ is a $\kappa$-{\em Aronszajn} tree if it is a $\kappa$-tree with no branch of size $\kappa$. If $\kappa = \aleph_1$ we just say Aronszajn tree.
\item
$T$ is a $(\kappa, {\leq}\lambda)$-Aronszajn tree if it is a tree of height $\kappa$ with each level of size ${\leq}\lambda$ and no branch of size $\kappa$. An $(\aleph_1, {\leq}\lambda)$-Aronszajn tree is called a {\em wide Aronszajn tree} if $\lambda$ is uncountable and the equality is witnessed at some level i.e. it is not a $(\omega_1, {<}\omega_1)$-tree\footnote{The use of the word ``wide" appears to come from the recent (and fascinating) paper \cite{DzSh2020}, though the concept has been in the literature for over 50 years. }. 
\item
A (wide) Aronszajn tree is {\em special} if it can be decomposed into countably many antichains. Equivalently if there is a {\em specializing function} $f:T \to \mathbb Q^+$, the set of positive rationals so that $f$ is strictly increasing on linearly ordered subsets of $T$.
\item
A tree $T$ of height $\omega_1$ (potentially with branches) is {\em essentially special} if there is a function $f:T \to \mathbb Q$ which is weakly increasing on chains and so that for all $s \leq_T t, u$ if $f(t) = f(s) = f(u)$ then $t$ and $u$ are comparable. 
\item
An $\omega_1$-tree is Kurepa if it has more than $\aleph_1$ many uncountable branches. It's a {\em weak} Kurepa tree if it is a tree of height and cardinality $\aleph_1$ with more than $\aleph_1$ many branches.
\end{enumerate}
\end{definition}
Throughout this note I will only be considering {\em normal} trees: that is a tree $T$ so that $|T_0| = 1$, every node is comparable with nodes on every level, and for each $s, t \in T$ of limit height $\alpha$, if $s\neq t$ there is a $\beta < \alpha$ so that $s \hook \beta \neq t \hook \beta$. Unless otherwise specified, in what follows ``tree" means ``normal tree".

Special trees were first investigated in connection with forcing in \cite{BMR70} where it was shown that the poset to add a specializing function with finite approximations is ccc and hence $\MA + \neg \CH$ implies that all trees of height $\aleph_1$, cardinality less than $2^{\aleph_0}$ and no uncountable branch are special. This poset obviously adds reals. Specializing without adding reals is more delicate as we will see.

\section{Specializing a Wide Tree}
In this section I work towards proving Theorem \ref{specwidetree}. The forcing notion used is very similar to the poset from \cite[Section 4]{AbrahamShelah93} which specializes a thin tree without adding reals. This is due to Abraham and Shelah, building on the original example of such a poset from \cite[Chapter V, Theorem 6.1]{PIPShelah}. Throughout, fix an $\omega_1$-tree $T$ (possibly with uncountable branches) and let $S \subseteq T$ be an $(\omega_1, {\leq}\omega_1)$-Aronszajn tree with the induced suborder. Without loss we may assume that $T \subseteq H_{\omega_1}$. The first step is to define the forcing $\mathbb P$. The idea is to force with partial specializing functions $f:S \to \mathbb Q$ but use the structure of $T$ to control the forcing. 

I begin by defining the objects that will build up the conditions. Throughout there is a subtlety concerning partial functions from $T$ to $\mathbb Q$ that I want to address up front. Fix ordinals $\alpha < \beta < \omega_1$. Often times I will be considering some function $h$ which maps some subset of $T_\beta$ to $\mathbb Q$ and we would like to consider the projection of this function to level $\alpha$ i.e. a new function $\hat{h}$ so that for each $t \in {\rm dom}(h)$ $\hat{h}(t \hook \alpha) = h(t)$. The issue is that $\hat{h}$ as written may not be a function since several different $t$'s on level $\beta$ may have the same projection to level $\alpha$. To avoid this I will use the following convention: let $\hat{h}(t \hook \alpha) = {\rm min} \{ h(s) \; | \; s \in {\rm dom}(h) \; {\rm and} \; s \hook \alpha = t \hook \alpha\}$ if this number exists and leave $\hat{h} (t)$ undefined otherwise. Note that if the domain of $h$ is finite the projections are always defined. When it will cause no confusion I will write $h \hook \alpha$ to denote the projection of $h$ to $\alpha$ and omit the hat.

\begin{definition}
Recall that $S \subseteq T$ are trees, $T$ is thin, potentially with cofinal branches and $S$ is wide, without cofinal branches. Throughout, unless otherwise noted, for a node $t \in S$, I mean by $t \hook \alpha$ the projection of $t$ to level $\alpha$ {\em in the sense of} $T$ (as opposed to $S$).

\begin{enumerate}
\item
A {\em partial specializing function} of height $\alpha$  is a function $f: T_{\leq \alpha} \cap S \to \mathbb Q$ which is strictly increasing on linearly ordered chains. We write $ht(f) = \alpha$ to denote the height of $f$.
\item
A (possibly partial) function $h:T_\beta \to \mathbb Q$ {\em projects into} $S$ if for each $t \in {\rm dom} (h)$ there is an $\alpha \leq \beta$ so that $t\hook \alpha \in S$. We say that such an $h$ {\em bounds} a partial specializing function if $\beta \geq \alpha + 1$ and for all $t$ in the domain of $h$ whose projection to the $\alpha + 1^{\rm st}$ level is in $S$ we have that $h(t) > f(t \hook \alpha + 1)$. 
\item
A {\em requirement} $H$ of height $\beta$ and arity $n = n(H) \in \omega$ is a countably infinite family of finite functions $h:T_\beta \to \mathbb Q$ which project into $S$ and whose domains have size $n$.
\item
A partial specializing function $f$ {\em fulfills} a requirement $H$ if the height of $f$ is at most the height of $H$ and for every finite $\tau \subseteq T_\beta$,  $\beta$ the height of $H$, there is an $h \in H$ bounding $f$ whose domain is disjoint from $\tau$.
\item
A {\em promise} is a function $\Gamma$ defined on a tail set of countable ordinals, the first of which we denote $\beta = \beta (\Gamma)$ so that for each $\gamma \geq \beta$, $\Gamma (\gamma)$ is a countable set of requirements of height $\gamma$ and if $\gamma ' \geq \gamma$ then $\Gamma (\gamma) = \Gamma (\gamma ') \hook \gamma$  i.e. every $H \in \Gamma(\gamma)$ there is some $H' \in \Gamma (\gamma ')$ so that $H ' = \{h \hook \gamma ' \; | \; h \in H\}$. Note that since each $h$ is finite, each projection $h \hook \gamma '$ is defined, however several distinct $h$'s may have the same projection.
\item
A partial specializing function $f$ {\em keeps} a promise $\Gamma$ if $\beta (\Gamma) \geq ht(f)$ and $f$ fulfills every $H \in \Gamma (\gamma)$ for all $\gamma \geq \beta$. Note that by the projection property given in the definition of a promise, to keep a promise it suffices to fulfill the requirements at the first level.
\item
The forcing notion $\mathbb P = \mathbb P_{T, S}$ consists of pairs $p = (f_p, \Gamma_p)$ where $f_p$ is a partial specializing function, $\Gamma_p$ is a promise and $f_p$ keeps $\Gamma$. We write $ht(p)$ for $ht(f_p)$ and $\beta(p)$ for $\beta(\Gamma_p)$. Finally we let $p \leq q$ if $f_p \supseteq f_q$, $\beta(p) > \beta(q)$ and for all $\gamma \geq \beta(p)$,  $\Gamma_p (\gamma) \supseteq \Gamma_q(\gamma)$. 
\end{enumerate}
\end{definition}

The proof of Theorem \ref{specwidetree} is broken up into a number of lemmas which collectively show that $\mathbb P$ has the properties advertized in the theorem. First let's show that any condition can be extended arbitrarily high up the tree. Note that this will imply that $\mathbb P$ specializes $S$.
\begin{lemma}
Suppose $p \in \mathbb P$ of height $\alpha$ and let $\beta \geq \alpha$. Then there is a $q \leq p$ of height $\beta$. Moreover, if $g: T_\beta \to \mathbb Q$ is a finite function bounding $f_p$ then $q$ can be found so that $g$ bounds $f_q$ as well.
\label{extensionlemma}
\end{lemma}

\begin{proof}
The proof is by induction on $\beta$. There are two cases.

\noindent {\bf\underline{Case I}:} $\beta$ is a successor ordinal. By induction assume that $\beta = \alpha + 1$. Thus there is a function $f_p:T_\alpha \to \mathbb Q$ and we need to find an $\widetilde{f}:T_\beta \to \mathbb Q$ so that $(f_p \cup \widetilde{f} , \Gamma_p \setminus \Gamma_p(\alpha)) \in \mathbb P$. I will define such an $\widetilde{f}$ in countably many stages, noting that there are only countably many things that we need to account for. Indeed I will define finite functions $f_n$ for $n < \omega$ so that $f_n \subseteq f_{n+1}$ and the union will be $\widetilde{f}$. Let $\{ (H_n, \tau_n) \; | \; n < \omega \}$ enumerate all possible pairs of requirements $H \in \Gamma_p(\beta)$ and finite subsets $\tau \subseteq T_\beta$. Also let $T_\beta \cap S = \{t_n \; | \; n < \omega\}$. Let $f_0$ be defined as follows: by the definition of $\mathbb P$, for $\tau = \tau_0 \subseteq T_\beta$ there is an $h \in H_0$ which projects into $S$ and so that $\tau_0$ is disjoint from the domain of $h$ and $h$ bounds $f_p$. For each $t \in {\rm dom} (h) \cap S$, let $f_0 (t)$ be any rational above $f_p(t \hook \alpha)$ less than the value of $h(t)$ and $g(t)$ if the latter is defined. Then, if $t_0 \notin {\rm dom}(h)$, define $f_0 (t_0)$ to be any value less than $g(t_0)$, again assuming this value is defined (again above $f_p$). This completes the construction of $f_0$. Note that its domain is finite. Suppose now that we have defined $f_i$ for all $i \leq n$ so that for all $i < n$ we have that $f_i \subseteq f_{i+1}$, $t_i$ (if it exists, note $T_\beta \cap S$ could be finite) is in the domain of $f_i$ and $f_i$ is bounded by some $h \in H_i$ whose domain is disjoint from $\tau_i$. Moreover assume that $f_i$ is bounded by $g$. Now I define $f_{n+1}$ by performing the same procedure as described for $f_0$, except that $\tau_{n+1}$ is replaced by $\tau_{n+1} \cup {\rm dom}(f_n)$. Note that this set is still finite so we can find a good $h \in H_{n+1}$ by the definition of a requirement. Now let $\widetilde{f} = \bigcup_{n < \omega } f_n$. Clearly this function is defined on all of $T_\beta$ and keeps the promise $\Gamma(\beta)$ so we are done.

\noindent {\bf\underline{Case II}:} $\beta$ is a limit ordinal. Fix a strictly increasing sequence $\langle \beta_n \; | \; n < \omega\rangle$ so that $\beta_0 = \alpha$ and ${\rm sup}_{n} \beta_n = \beta$. The idea is to weave the procedure described in Case I to build a function on $T_\beta \cap S$ with the inductive assumption that allows us to extend $f_p$ to each $\beta_n$. More concretely, as before  let $\{ (H_n, \tau_n) \; | \; n < \omega \}$ enumerate all possible pairs of requirements from $\Gamma_p(\beta)$ and finite subsets $\tau \subseteq T_\beta$. Also let $T_\beta \cap S= \{t_n \; | \; n < \omega\}$. Define $f_0$ as in Case I:  for each $t \in {\rm dom}(h) \cap S$, let $f_0 (t)$ be any rational above $f_p(t \hook \alpha)$ less than the value of $h(t)$ and $g(t)$ if the latter is defined. Then, if $t_0 \notin {\rm dom}(h)$, define $f_0 (t_0)$ to be any value less than $g(t_0)$, again assuming this value is defined (again above $f_p$). Now, using the inductive assumption, let $\hat{f}_1 \supseteq f_p$ be a partial specializing function of height $\beta_1$ bounded by $g \cup f_0$.

Now inductively suppose we have defined $f_i$ for all $i \leq n$ so that for all $i < n$ we have that $f_i \subseteq f_{i+1}$, $t_i$ is in the domain of $f_i$ (if $t_i$ exists) and $f_i$ is bounded by some $h \in H_i$ whose domain is disjoint from $\tau_i$. Moreover assume that $f_i$ is bounded by $g$. Also suppose that we have defined $\hat{f}_i: T_{\beta_i} \cap S \to \mathbb Q$ for all $i \leq n$ so that for all $i < n$ $\hat{f}_{i+1} \supseteq \hat{f}_i$ and $\hat{f}_n$ is bounded by $g \cup f_n$. Now define $f_{n+1}$ by performing the same procedure as described for $f_0$, except that $\tau_{n+1}$ is replaced by $\tau_{n+1} \cup {\rm dom}(f_n)$ and $f_p$ is replaced with $\hat{f}_n$. Then we define $f_{n+1}$ exactly as in Case I, using $\hat{f}_n$ in placed of $f_p$ and then extend $\hat{f}_n$ to a function $\hat{f}_{n+1} : T_{\beta_{n+1}} \cap S \to \mathbb Q$ bounded by $g \cup f_{n+1}$ via the inductive assumption. 

Finally let $f_q = \bigcup_{n <  \omega} \hat{f}_n \cup f_n$. This function is then defined on all of $T_{ \leq{\beta}}$ and keeps all requisite promises and is bounded by $g$ so we're done. 
\end{proof}

Next I show how to add promises.  Given two promises $\Gamma$ and $\Psi$ I write $\Psi \subseteq \Gamma$ to mean that $\beta(\Gamma) \geq \beta(\Psi)$ and for all $\gamma \geq \beta (\Gamma)$ we have that $\Psi(\gamma) \subseteq \Gamma (\gamma)$. Also, I will write $\Gamma \cup \Psi$ to mean the promise $\Delta$ so that $\beta(\Delta) = {\rm max}\{\beta(\Gamma), \beta(\Psi)\}$ and for all $\gamma \geq \beta(\Delta)$ $\Delta(\gamma) = \Gamma(\gamma) \cup \Psi(\gamma)$.

\begin{lemma}
Suppose $p \in \mathbb P$ is of height $\alpha$, $\beta \geq \alpha$ and $g:T_\beta \to \mathbb Q$ is a finite function bounding $f_p$. Let $\Psi_g$ be a promise so that $\beta(\Psi_g) \geq \beta$ and for all $H \in \Psi_g (\beta(\Psi_g))$ if $h \in H$ then for each $t \in {\rm dom}(h)$ $t \hook \beta$ is in the domain of $g$ and $ h(t) \geq g(t \hook \beta)$. Then there is an extension $q \leq p$ so that $\Gamma_q \supseteq \Psi_g$. Moreover $q$ can be chosen to have any height greater than or equal to $\alpha$.
\label{basis}
\end{lemma}

Following \cite[Lemma 4.4]{AbrahamShelah93} we refer to the $g$ in the above lemma as a {\em basis} for the promise $\Psi_g$ and say that $g$ {\em generates} $\Psi_g$.

\begin{proof}
Note that the promise $\Psi_g$ is constructed so that any condition bounded by $g$ keeps it. Thus in fact $q = (f_p, \Gamma_p \cup \Psi_g)$ is as needed. For the moreover part, now use Lemma \ref{extensionlemma} to strengthen further.
\end{proof}

Intuitively the previous lemma states that if $g$ bounds some condition, it's not dense to insist that extensions are not bounded by $g$. In particular, we can always avoid growing above $g$. This is key for the proof of properness and it's from here that the need for promises stems. To prove that $\mathbb P$ is proper I will need the following lemma. 

\begin{lemma}
Let $\theta$ be sufficiently large and let $M \prec H_\theta$ be countable containing $T, S, \mathbb P, etc$. Let $p \in \mathbb P \cap M$ and let $\delta = M \cap \omega_1$. Note that $M \cap T = T_\delta$. Let $D \in M$ be a dense open subset of $\mathbb P$ and let $h:T_\delta \to \mathbb Q$ be a finite function bounding $f_p$. Then there is an extension $q \in D \cap M$ so that $f_q$ is also bounded by $h$.
\label{submodel}
\end{lemma}

Implicit in the proof below is the following fact which follows from elementarity: for any $\theta$ sufficiently large, and $M \prec H_\theta$ with $S, T \in M$, $S \cap M$ is unbounded in $T \cap M$.

\begin{proof}
Suppose the statement of the lemma is false and let $M$, $T$, $S$, $p$, $D$, $h$ etc be a counter example. Let me fix that $ht(p) = \alpha$. Note that if $q \leq p$ and $q \in D \cap M$, then $q$ is not bounded by $h$. I will reach a contradiction by showing how to add a promise to $p$ as in Lemma \ref{basis} which ensures that {\em any} further extension is bounded by $h$. Let us enumerate the domain of $h$ by $t^h_0, ..., t^h_{n-1}$. Also, without loss, assume that all projections of each $t^h_i$ project into $S$ since otherwise they do not matter. Since $T$ is normal, there is a least level $\gamma > \alpha$ so that for all $i < j < n$ the projections $t^h_i \hook \gamma$ and $t^h_j \hook \gamma$ {\em in the sense of $S$} (!!) are distinct. Let $h_\gamma$ be the projection of $h$ to this level. Note that $h_\gamma \in M$ (since it's finite), bounds $f_p$ and if $q \leq p$ is in $M \cap D$ then $q$ must not be bounded by $h_\gamma$ since otherwise we would contradict our assumption. Thus we obtain that $M \models `` \forall q \leq p \, {\rm if} \; q \in D \; {\rm then} \; q \; {\rm is \; not \; bounded \; by} \; h_\gamma$". Note by elementarity this is also true in $V$. Note also, that since the fact that $\gamma$ was least was not used, we could have chosen {\em any} $\gamma ' \geq \gamma$ and the statement above would have held with $\gamma$ replaced by $\gamma '$. Hence $M$ thinks there are cofinally many $\gamma '$ below $\delta$ so that $h_{\gamma '}$ is as in the statement $M$ models above (this uses the fact mentioned before the start of this proof). 

Now, I want to use the property described of $h_\gamma$ to define a collection of $n$-tuples of $S$. Let's say that an $n$-tuple $\vec{s} = \langle s_0, ..., s_{n-1}\rangle$ of elements of $S$, all of the same height $\geq \gamma$, is {\em bad} if its projection to level $\gamma$ is ${\rm dom} (h_\gamma)$ (say $s_i \hook \gamma = t_i^h \hook \gamma$) and the function $h_{\vec{s}}$ whose domain is $\vec{s}$ and, to each $s_i$ assigns the rational $h(t^h_i)$ bounds $f_p$ but is such that there is no $q \leq p$ so that $q \in D$ and $f_q$ is bounded by $h_{\vec{s}}$. In particular the domain of $h_{\gamma '}$ is bad for each $\gamma ' \geq \gamma$, but other sets, which are not the projection of the domain of $h$ may also be bad. Let $B \subseteq S^n$ be the collection of all bad tuples. For simplicity we let $B(i)$ be the collection of all bad tuples on level $i < \omega_1$. Note that since $M \models ``\{ i < \omega_1 \; | \; B(i) \neq \emptyset\} {\rm \; is \; unbounded}"$ this is true in $V$. Also $B$ is closed downwards above $\gamma$ in the sense that if $\vec{s} \in B(j)$ and $\gamma < i < j$ then $\vec{s} \hook i \in B(j)$. If $\vec{s}_0, \vec{s}_1 \in S^n$ let's write $\vec{s}_0 \leq \vec{s}_1$ if the elements of $\vec{s}_0$ are componentwise below the elements of $\vec{s}_1$ on the tree. We define recursively $B_0 = B$, $B_{i+1} = \{\vec{s} \in B_i \; | \; {\rm for \; uncountably \; many \; levels \; } j \exists \vec{s}' \in B_i(j) \, \vec{s} \leq \vec{s}'\}$, and $B_\lambda = \bigcap_{i < \lambda} B_i$ for $\lambda$ limit. Note that ${\rm dom}(h_\gamma) \in B_i$ for every $i$. Also observe by construction that if $i \leq j$ then $B_i \supseteq B_j$. Let $B_\infty = B_\rho$ where $\rho$ is the least so that $B_\rho = B_{\rho + 1}$. 

\begin{claim}
Every $\vec{s} \in B_\infty$ has two disjoint extensions under $\leq$ in $B_\infty$.
\end{claim}

This is essentially \cite[Lemma 16.18]{JechST}, modified to the current context.
\begin{proof}[Proof of Claim]
Suppose not and let $\vec{s} \in B_\infty$ be a counter example. I will use $\vec{s}$ to define a branch through $S$ contradicting the fact that $S$ is Aronszajn. Let $W \subseteq B_\infty$ be the collection of all $\vec{u}$ extending $\vec{s}$ and for each $i < \omega_1$ let $W(i)$ denote the set of tuples $\vec{z} \in W$ of height $i$. Since every element of $B_\infty$ has extensions on cofinally many levels, $W$ has elements on all levels above the height of $\vec{s}$. Let's denote this height by $\gamma_s$. Finally note that since we're assuming that $\vec{s}$ does not have disjoint extensions, given any $\gamma_s < i < j < \omega_1$ we have that if $\vec{z}^i \in W(i)$ and $\vec{z}^j \in W(j)$ then it must be that there is a $k, k' < n$ for which $z^j_k \hook i = z^i_{k'}$.

Let $U$ be an ultrafilter on $W$ all of whose elements contains tuples unboundedly high up in $S$. For any $x \in S$ and $k < n$ let $Y_{x, k}$ be the collection of all elements $\vec{z} \in W$ so that $x$ is comparable with the $k^{\rm th}$ element of $\vec{z}$. Notice by the above assumption, we get that $W = \bigcup_{l < n}\bigcup_{k < n} Y_{z^i_l, k}$ where $\vec{z}^i \in W(i)$ for any $i \in (\gamma_s, \omega_1)$. Since $U$ is an ultrafilter, for any such $i$ we must have that there is an $l_i < n$ and a $k_i < n$ so that $Y_{z^i_{l_i}, k_i} \in U$. But then for some $k$ the set $I = \{ i \in (\gamma_s, \omega_1) \; | \; k_i = k\}$ is uncountable. I claim that $z^i_{l^i} \leq_S z^j_{l^j}$ for any $i< j \in I$. To see this, note that since $Y^i_{z^i_{l_i}, k} \cap Y^i_{z^j_{l_j}, k} \in U$ so there is a $\vec{z} \in W$ of height $\lambda$ in this intersection for some $j < \lambda$ and hence $z^i_{l_i}, z^j_{l_j} \leq_S z_k$ so $z^i_{l_i}, z^j_{l_j}$ are comparable. But now the set $\{z^i_{l_i} \; | \; i \in I\}$ must generate a cofinal branch in $S$, contradiction.
\end{proof}

By bootstrapping the above argument, there is a level $i$ so that $B_\infty (i)$ has infinitely many disjoint bad tuples. Let $i$ be such a level. But now such a level generates a promise $\Psi$ whose basis is $h_\gamma$: %we simply take the projections of the nodes (which are all in the same level in $S$) to the least level in $T$ in which all such projections are defined. 
Such a promise is in $M$, by running the argument above in $M$, and moreover, $p$ keeps this promise so we can add it to $p$ (in $M$). Concretely, the promise $\Psi$ is defined by looking at the set $\{h_{\vec{s}} \; | \; \vec{s}{\rm \; projects \; into} \; A\}$ for $A\subseteq B_\infty(i)$ an infinite pairwise disjoint set and considering requirements at all higher levels projecting to this. Therefore there is a $p' \leq (f_p, \Gamma_p \cup \Psi)$ in $M$ by Lemma \ref{basis}. But now let $q \leq p'$ be any element in $D \cap M$. Then $q$ keeps the promise $\Psi$ but this contradicts the definition of a bad function. 
\end{proof}

\begin{lemma}
$\mathbb P$ is proper. In fact, $\mathbb P$ is dee-complete for some simple completeness system $\mathbb D$.
\end{lemma}

\begin{proof}
Work in the setting of Lemma \ref{submodel}. I want to prove the existence of a master condition for $M$. Let $\langle D_n \; | \; n < \omega \rangle$ be an enumeration of the dense open subsets of $\mathbb P$ in $M$. Let $p \in \mathbb P \cap M$ and let $\langle t_i \; | \; i < \omega \rangle$ enumerate the elements of $T_\delta$ which project into $S$ and let $\langle \tau_k \; | \; k < \omega\rangle$ enumerate all the finite subsets of $T_\delta$. I want to define a sequence $p \geq p_0 \geq p_1 \geq ... \geq p_{n} \geq ...$ so that $p_i \in D_i \cap M$ for all $M$ and there is a condition $q$ extending the union of the $p_i$'s. Such a $q$ defines a generic over $M$. The idea is to use Lemma \ref{submodel} $\omega$-many times to make sure that the union of an $M$ generic filter is bounded and hence can be extended into a further condition. I will then extract from the proof a definition of the generics bounded by such a $q$ and this will be used to define a simple completeness system as needed. 

Fix an enumeration in order type $\omega$ of all triples $e_l = (m_l, n_l, k_l)$ so that $m_l, n_l, k_l \in \omega$ and the first occurrence of $m$ in the first coordinate is after the $m^{\rm th}$ element of the enumeration and each such triple appears infinitely often. Now, using Lemma \ref{submodel}, recursively define conditions $p_{i+1}$ and functions $h_i$ satisfying the following conditions:

i) $p_{i+1} \leq p_i$ and $p_{i+1} \in D_{i+1}$ so that we ensure that $p_{i+1}$ fulfills the $n_i^{\rm th}$ requirement of $p_{m_i}$ with respect to $\tau_{k_i}$ (this uses Lemma \ref{submodel}). Let $t_{j_{i+1}}$ be the node of the tree that we bound $p_{i+1}$ by in this step. We can assume inductively that $t_{j_{i+1}}$ is not in the domain of $h_i$.

ii) $h_i$ has a finite domain consisting of all $t_0, ..., t_i$ and $t_{j_0}, ..., t_{j_i}$ which project into $S$, bounds $p_i$ and is at most $(f_{p_i}(t_{j_i} \hook ht(p_i)) + h_{t_{j_i}})/2$ where $h_{t_{j_i}}$ is the rational in the range of the $n^{\rm th}_i$ requirement of $\Gamma_{p_i} (\delta)$ which corresponds to $t_{j_i}$ and was chosen in the $i^{\rm th}$ step of the process.

iii) $h_{i+1}$ is a finite function from $T_\delta$ to $\mathbb Q$ bounding $f_{p_{i+1}}$ which extends $h_i$ to include in its domain $t_{i +1}$ and $t_{j_{i+1}}$ if these project into $S$ (note potentially these nodes are the same).

It's clear by what we have done that such a sequence can be constructed and generates a generic filter on $M$. I need to show that there is a lower bound, $q$. Note that $\bigcup_{n < \omega} f_{p_n}$ is a partial specializing function defined on $T_{ < \delta}$. I claim that we can extend it to a function defined on $T_\delta$ which keeps the promises $\bigcup_{n < \omega} \Gamma_{p_n}$. Indeed, let $q(t_i) = h_i(t_i)$. This is defined, since we insisted that $t_i \in {\rm dom}(h_i)$. Also, since $h_i$ bounded all $p_i$, $q(t_i)$ is at least the supremum of the values of $f_n (t_i \hook \beta)$ for all $\beta < \delta$. What needs to be checked is that $f_q$ actually keeps all the promises in the $p_i$'s. This is what was planned for though. If $H \in \Gamma_q(\delta)$ then $H \in \Gamma_{p_i} (\delta)$ for some $i$ and for any $\tau \subseteq T_\delta$ finite, there was a stage where we ensured that $f_q$ was bounded by some $h_n$ which included being bounded by some $H$ on a node disjoint from $\tau$. Then, from that stage on, since all $p_j$'s were bounded by this $h$, we get that $f_q$ keeps that instance of the promise.

Thus we have shown that $q$ is an $(M, \mathbb P)$-master condition so $\mathbb P$ is proper. It remains to show that it is in fact dee-complete. To do this, we need some simple way of coding generics of countable models with lower bounds such as those found in the previous paragraph. To this end, I start by defining the types of codes that are needed. Given a sufficiently large $\theta$ and a countable transitive $\sigma:\overline{M} \prec H_\theta$ so that $\sigma (\overline{\P}, \bar{T}, \bar{S}) = \P, T, S$, let us say that an element $u \in H_{\omega_1}$ {\em codes a suitable witness} for $\overline{M}$ and $\overline{\P}$ if $u$ codes a pair $\langle B, c\rangle$ so that:
\begin{enumerate}
\item
$B$ is a countable set of branches through $\overline{M} \cap T = T_\delta$ with $\delta = (\omega_1)^{\overline{M}}$ which intersect $S$ (so $B \subseteq T_\delta$ and hence can be coded by a real). 
\item
$c: \overline{\P} \to \mathcal P(\overline{M})$ is a function so that $c(\overline{p})$ is a countable set of requirements of height $\delta$ which are fulfilled by $f_{\bar{p}}$ and project to the promise $\Gamma_{\bar{p}}$. In other words, for each $\alpha < \delta$ and each $\bar{H} \in \Gamma_{\bar{p}}(\alpha)$ there is a $H \in c(\overline{p})$ so that $\bar{H} = \{h \hook \alpha \; | \; h \in H\}$. 
\end{enumerate}
Note that coding a suitable witness is definable in $H_{\omega_1}$. Now, if $\overline{M} \prec H_\theta$ and $u = \langle B, c\rangle$ codes a suitable witness for $\overline{M}$ and $\P$ and $\bar{p} \in \overline{\P} \cap \overline{M}$ let us say that a $\overline{\mathbb P}$-generic $\bar{G} \ni \bar{p}$ over $\overline{M}$ is {\em good} for $u$ if there is a function $f:B \to \mathbb Q$ so that $\bigcup_{\bar{q} \in \bar{G}} f_{\bar{q}} \cup \{f\}$ is a partial specializing function (restricted to the elements of $T_\delta \cap S$ determined by $B$) satisfying the requirements $c(\bar{p})$. Finally if $M = \sigma `` \overline{M}$ then a $\mathbb P$-generic $G$ over $M$ is {\em good for} $u$ if $G$ is generated by $\sigma `` \bar{G}$ for a $\bar{G}$ which is good for $u$. Note that if $u = \langle B, c\rangle$ is such that $B$ contains the set of branches with upperbounds in $T_\delta \cap S$ and $G$ is good for $u$ then $G$ has a lower bound: the pair consisting of the partial specializing function $\bigcup_{\bar{q} \in \bar{G}} \sigma(f_{\bar{q}}) \cup \{f\}$ and the promise generated by $c(\bar{p})$. Also, good generics exist for every condition and model by the argument for properness in the first half of this proof.

Finally we can define our completeness system by letting $\mathbb D(\overline{M}, \overline{\mathbb P}, \bar{p})$ be the set of $A_u$ for $u \in H_{\omega_1}$ where if $u$ codes a suitable witness for $\overline{M}$ and $\overline{\mathbb P}$ then $A_u$ is the set of generics which are good for $u$ and if $u$ does not code a suitable witness then $A_u$ is all generics. This is definable and satisfies the conditions of a completeness system. The only thing that is not immediately clear is the countable closure. This is why promises consist of countable sets of requirements: Suppose that $\{\langle B_i, c_i\rangle \; | \; i < \omega\}$ all code suitable witnessess, $B = \bigcup_{i < \omega} B_i$ and let $c$ be the function sending $\bar{p} \mapsto \bigcup_{i < \omega} c_i(\bar{p})$. Then $u = \langle B, c \rangle$ is a code for a suitable witness and any generic that is good for $u$ is good for all $\langle B_i, c_i\rangle$ hence $\bigcap_{i < \omega} A_{\langle B, c_i\rangle}$ is nonempty.

%This amounts to showing that we can define the collection of filters of the form $G_q = \{ p \in \mathbb P \cap M \; | \; q \leq p\}$. Clearly such filters can be recovered from the choice of the enumeration of the nodes of $T_\delta$, the requirements, $\langle e_l \; | \; l < \omega\rangle$, and the choice of the $h_i$'s. This set of parameters is hereditarily countable so it's coded by a single element of $H_{\omega_1}$. Fix such a code $s$ and let $\phi(G, x, N, \mathbb P, p, s)$ be the formula which says that \say{$G$ is $\mathbb P$-generic over $N$ (in fact the Mostowski collapsed versions of these), $p \in G$ and either $x$ codes the same enumerations as $s$, except maybe the functions $h_i$ coded by $x$ are larger than those coded by $s$ or else we take all generics}. Then clearly this is a well defined simple completeness system. To check that countable intersections are non-empty, note that for any $x$ any $A_x$ contains the generic bounded by the $h_i$'s coded by $s$.

\end{proof}

Finally I prove that $\mathbb P$ is $\alpha$-proper for all $\alpha < \omega_1$. 
\begin{lemma}
Let $\alpha < \omega_1$ and let $\vec{N} = \langle N_i \; | \; i \leq \alpha\rangle$ be a tower of length $\alpha$ for $N_i \prec H_\theta$, $\theta$ sufficiently large with $\mathbb P \in N_0$. Then for any $p \in N_0 \cap \mathbb P$ there is a $q \leq p$ which is $(N_i, \mathbb P)$-generic simultaneously for every $i \leq \alpha$.
\end{lemma}

\begin{proof}
If $\alpha$ is a successor ordinal, this is just the proof of properness given above so assume that $\alpha$ is a limit ordinal. Pick an increasing sequence $\langle \alpha_n \; | \; n < \omega \rangle$ with ${\rm sup}_{n < \omega} \alpha_n = \alpha$. Let $\delta = \omega_1 \cap N_\alpha$. One can perform the same proof as when it was proved that $\mathbb P$ was proper, except now we insist (via the inductive assumption) that $p_i$ be $(N_j , \mathbb P)$-generic for all $j < i$ and $p_i \in N_i$ as opposed to $p_i$ being in some specified dense open. Since, by the definition of a tower $\langle N_j \; | \; j < i \rangle \in N_i$ this is possible (given the sequence, by elementarity, $N_i$ can find a master condition). Moreover, since, again by definition of a tower, the sequence of models is continuous and in particular, $N_\alpha = \bigcup_{n < \omega} N_{\alpha_n}$ the set $\{r \in N_\alpha \cap \mathbb P \; |\; \exists i \, p_i \leq r\}$ is $(N_\alpha, \mathbb P)$-generic. The only thing to be careful about is that the union of the $p_i$'s can be extended to some $q$ of height $\delta$. However, by iteratively applying Lemma \ref{submodel} as in the previous proof this is easily accounted for.
\end{proof}

Therefore $\mathbb P$ is dee-complete and ${<}\omega_1$-proper, thus proving Theorem \ref{specwidetree}. We get as an immediate corollary the following.
\begin{corollary}
Assume $\DCFA$. Every wide Aronszajn tree which embeds into an $\omega_1$-tree is special.
\end{corollary}

We can also iterate this forcing with countable support to obtain the following (with no consistency strength). Note that under $\CH$ the forcing notion $\mathbb P$ has the $\aleph_2$-c.c. since any two conditions with the same partial specializing function are compatible. 

\begin{corollary}
It's consistent with $\CH$ that all wide Aronszajn trees which embed into an $\omega_1$-tree are special.
\end{corollary}

I will give a concrete application of such a tree in the next section. Let me note first that the condition is not trivial: there are wide Aronszajn trees in $\ZFC$ which cannot be embedded into $\omega_1$ trees.

\begin{lemma}(Essentially Todor\v{c}evi\'c, see \cite[Definition 3.2]{Todorcevic1981})
There is an $(\omega_1, {\leq}2^{\aleph_0})$-Aronszajn tree which is $\ZFC$-provably non-special, and cannot be specialized by any forcing not adding reals.
\label{lemmaT(S)}
\end{lemma}

\begin{proof}
Let $E \subseteq \omega_1$ be stationary co-stationary and let $T(E)$ be the tree of attempts to shoot a club through $E$. In other words, elements of $T$ are closed, bounded, countable initial segments of $E$ ordered by end extension. This poset is well known to be $\sigma$-distributive, hence the tree has height $\aleph_1$. Also, every element is a countable set of ordinals hence it can be coded by a real and therefore the tree has width $2^{\aleph_0}$. So we conclude that $T(E)$ is an $(\omega_1, {\le}2^{\aleph_0})$-Aronszajn tree. However, it can't be special, since, as mentioned before, forcing with this tree does not add reals, so in particular, $\omega_1$ is preserved. To see that it remains non-special in every forcing extension not adding reals, note that, if $\mathbb P$ does not add reals then the reinterpretation of $T(E)$ in $V^\mathbb P$ is just $\check{T(E)}$ so it's still $\sigma$-distributive and hence it must still not be special.
\end{proof}

Putting together this lemma and Theorem \ref{specwidetree} we conclude the following odd result which may be of independent interest. Note that the theorem below is provable in $\ZFC$.

\begin{theorem}
For any stationary $E \subseteq \omega_1$ the tree $T(E)$ cannot be embedded into any $\omega_1$-tree.
\label{embedthm}
\end{theorem}

\begin{proof}
Suppose $T(E)$ could be embedded into an $\omega_1$-tree. Then, by forcing with the forcing from Theorem \ref{specwidetree} we could make $T(E)$ special without adding reals. But this contradicts Lemma \ref{lemmaT(S)}.
\end{proof}

\begin{corollary}
$\DCFA$ is consistent with the existence of non-special trees of size $\aleph_1$.
\label{cornonspec}
\end{corollary}

\begin{proof}
If $\CH$ holds, which it does in the natural model of $\DCFA$, then the tree $T(E)$ witnesses the corollary.
\end{proof}

\section{Kurepa Trees}

I now use the forcing from the previous section to provide an application of $\DCFA$. First let me note the following, which is Theorem 7.4 of \cite{BaumPFA} coupled with the remarks preceeding its statement on page 949 of the same article.
\begin{lemma}
Every essentially special tree has at most $\aleph_1$ many cofinal branches.
\end{lemma}

\begin{proof}
Suppose $T$ is an essentially special tree as witnessed by $f:T \to \mathbb Q$. Let $\mathcal B$ be the set of uncountable branches through $T$. From $f:T \to \mathbb Q$ we can define an injection $g:\mathcal B \to T$ as follows. For each $b \in \mathcal B$ by pigeonhole there is some $r \in \mathbb Q$ so that $\{t \in b \; |\; f(t) = r\}$ is cofinal in $b$. Pick such an $r$ and let $g(b)$ be the least $t \in b$ with $f(t) = r$ (or indeed any such $t$). By the definition of $f$, if $b_1 \neq b_2$ then $g(b_1) \neq g(b_2)$. To see this, suppose that $g(b_1) = g(b_2) = s$, let $f(s) = r$ and let $t \in b_1 \setminus b_2$ with $f(t) = r$ and $u \in b_2 \setminus b_1$ with $f(u) = r$. Such $t$ and $u$ exist by the assumption on $r$. But this is a contradiction since we have that $s \leq_T t, u$ with $f(s) = f(t) = f(u)$ and $t$ and $u$ are incomparable. Therefore $g$ is an injection from $\mathcal B$ into $T$ so $|\mathcal B| \leq \aleph_1$.
\end{proof}

\begin{theorem}
Under $\DCFA$ all $\omega_1$-trees are essentially special and hence there are no Kurepa Trees.
\label{DCFAimpliesnotKH}
\end{theorem}

Note that \cite[Chapter VII, Application G]{PIPShelah} proves the same thing under the additional assumptions that $2^{\aleph_0} = \aleph_1$ and $2^{\aleph_1} = \aleph_2$. What is new is that the cardinal arithmetic is unnecessary. The proof also gives more information since it shows that certain wide trees are special under $\DCFA$. The proof of Theorem \ref{DCFAimpliesnotKH} follows Baumgardner's original proof from $\PFA$, however using the poset from Theorem \ref{specwidetree}. The argument is sketched with the reader referred to \cite[Section 7]{BaumPFA} for more details. %We include a detailed proof for completeness.

\begin{proof}
Assume $\DCFA$ and let $T$ be an $\omega_1$-tree. Let $\lambda \geq \aleph_2$ be the number of branches through $T$. First, force with $Col (\lambda, \aleph_1)$, the $\sigma$-closed forcing to collapse $\lambda$ to $\aleph_1$. Note that, being $\sigma$-closed, this is dee-complete and ${<}\omega_1$-proper. Work in the collapse extension. As noted in Lemma 7.11 of \cite{BaumPFA} $\sigma$-closed forcing won't add uncountable branches to a tree of width $<2^{\aleph_0}$ hence, in particular, there are no new branches added to to $T$ in the extension so there are now at most $\aleph_1$ many branches.

I use the following claim, due to Baumgartner, see \cite[Lemma 7.7]{BaumPFA}.
\begin{claim}
There is an uncountable subtree $S \subseteq T$ with no uncountable branches and, by specializing it, we obtain that $T$ is essentially special i.e. there is an $f:T \to \mathbb Q$ which is (weakly) increasing on chains and for all $s, t, u \in T$ if $s \leq_T t, u$ and $f(s) = f(t) = f(u)$ then $t$ and $u$ are comparable.
\end{claim}

%\begin{proof}
%Let $\mathcal B$ denote the set of uncountable branches through $T$. By the remark preceding the claim we have that $|\mathcal B| = \aleph_1$. By Lemma 7.6 of \cite{BaumPFA} there is an injection $g:\mathcal B \to T$ so that for each $b \in \mathcal B$ $g(b) \in b$ and whenever $g(b_1) <_T g(b_2)$ then $g(b_2) \notin b_1$. Now let $S = \{t \in T \; | \; \forall b \in \mathcal B \, {\rm if} \; t \in b \, {\rm then} \; t \leq_T g(b)\}$. This is a tree with the order inherited from $T$. Moreover, it's uncountable since it countains the range of $g$. It remains to see that this tree has no uncountable branches.

%Towards a contradiction, suppose that $b$ were an uncountable branch through $S$. Let $\bar{b} = \{t \in T \; | \; \exists s \in b \; t <_T s\}$ i.e. the downward closure of $b$ in $T$. This must be an uncountable branch through $T$. But then since $g(\bar{b}) \in \bar{b}$ we get an $s \in b$ with $g(\bar{b}) \leq_T s$ contradicting the definition of $S$.
%\end{proof} 

Thus, by applying the specializing forcing $\mathbb P_{T, S}$ from Theorem \ref{specwidetree} to $S$ and working in that extension we have that $S$ is special and so $T$ is essentially special. Now, finally applying $\DCFA$ we can pull back to $V$ and find an $f:T \to \mathbb Q$ witnessing that $T$ is essentially special, so we're done. %Let $f:S \to \mathbb Q$ be such a specializing function. Let $t\in T \setminus S$. We extend $f$ to include $t$ as follows. Since $t \notin S$ there is a branch $b$ so that $t \in b$ but $g(b) \leq_T t$. This branch is unique: If $g(b_1) <_T g(b_2) <_T t$ with $t\in b_1 \cap b_2$ then in particular $g(b_2) \in b_1$ which contradicts the choice of $g$. Now let $f(t) = f(g(b))$ for this unique branch. 

\end{proof} 

%An $\omega_1$-tree $T$ is called {\em essentially special} if there is an $f:T \to \mathbb Q$ which is (weakly) increasing on chains and for all $s, t, u \in T$ if $s \leq_T t, u$ and $f(s) = f(t) = f(u)$ then $t$ and $u$ are comparable. The above proof actually shows the following.

%\begin{theorem}
%Under $\DCFA$ all $\omega_1$-trees are essentially special. 
%\end{theorem}

In contrast with the case of $\PFA$, note that Corollary \ref{cornonspec} this result cannot be improved to trees of width $\omega_1$. 

\section{Cardinal Characteristics and Open Problems}

The previous sections suggest some new directions for studying wide trees, particularly in connection with cardinal characteristics. While I leave an indepth investigation of these ideas for future research I want to finish this note by recording some easy observations and connecting them back to what has been shown.

The main observation is that the behavior of trees is as much connected to their width and cardinality as to their height. This is obscured by the fact that the ccc forcing to specialize a tree works equally well regardless of the width of the tree. However, the trees of the form $T(E)$ suggest that there is something more subtle going on with regards to specializing wider trees. The following cardinals attempt to measure this.

\begin{definition}
\begin{enumerate}
\item
$\mathfrak{st}$, the $\mathfrak{s}$pecial $\mathfrak{t}$ree number, is the least cardinal $\lambda$ such that there is a non-special $(\omega_1, {\leq}\lambda)$-Aronszajn tree of cardinality $\lambda$.
\item
$\mathfrak{no}$, the $\mathfrak{no}$ new reals number, is the least cardinal $\lambda$ of an $(\omega_1, 
{\leq}\lambda)$-Aronszajn tree of cardinality $\lambda$ which can be forced to be special without adding reals.
\end{enumerate}
\end{definition}

Let's make some easy observations.
\begin{observation}
$\aleph_1 \leq \mathfrak{st} \leq \mathfrak{no}\leq \mfc$
\end{observation}

\begin{proof}
That $\mathfrak{st}$ is uncountable is essentially by definition. To see that $\mathfrak{st}\leq \mathfrak{no}$ it suffices to note that any special tree is obviously specializable without adding reals (by trivial forcing). Finally Todorcevic's tree $T(E)$ defined above witnesses that there is always a tree of size continuum that cannot be specialized without adding reals.
\end{proof}

I do not know exactly what these cardinals can be. It's clear that $\mathfrak{st}$ can remain $\aleph_1$ in models where many other cardinal characteristics are big since nearly all known cardinal characteristics can be made to have size continuum ($\neq \aleph_1$) while preserving the existence of a Souslin tree since we can make all cardinals (except $\mathfrak{m}$) in the Cicho\'n and Van Douwen diagrams large using $\sigma$-linked forcing. The following however is less clear.
\begin{question}
What provable bounds exist between known cardinal invariants and $\mathfrak{st}$? For instance, is it provable that $\mathfrak{st} \leq \mfd$? 
\end{question}
The number $\mathfrak{no}$ seems even more mysterious. I do not even know if it can consistently be less than the continuum.
\begin{question}
Is it consistent that $\mathfrak{no} < \mfc$? Is it consistent that $\mathfrak{st} < \mathfrak{no}$?
\end{question}
A potentially easier question, for which I conjecture the answer is ``yes" is the following:
\begin{question}
Does $\DCFA$ imply that $\mathfrak{no} = \mfc$?
\end{question}

Finally let me ask about extensions of the main theorem of this note.
\begin{question}
Are there (in $\ZFC$) trees which can be specialized without adding reals but are not embeddible into $\omega_1$-trees?
\end{question}

The use of forcing notions which specialize wide trees is key in several important applications of $\PFA$ including failure of various square principles, and the tree property on $\omega_2$. Therefore a natural question is whether the forcing $\mathbb P_{T,S}$ can be substituted in in these arguments.
\begin{question}
What other consequences of $\DCFA$ (possibly with some additional cardinal arithmetic assumption) can be obtained using $\mathbb P_{T,S}$? Does $\DCFA + \neg \CH$ imply the tree property on $\omega_2$? Does it imply the failure of weak square on $\omega_1$?
\end{question}

\noindent {\bf Acknowledgements} \; \; This reasearch was supported by a CUNY mathematics fellowship and the author would like to thank the mathematics department at the Graduate Center at CUNY for this. Some of the results here also appear in a different form as part of the author's PhD thesis. The author would like to thank his advisor, Professor Gunter Fuchs for suggesting the problem and many helpful discussions.

This paper was intended for the conference procedings of the RIMS Set Theory 2019 conference organized by Professor Daisuke Ikegami. The author would like to thank Professor Ikegami for organizing such a wonderful conference. Also, the author would like to thank Professor Hiroshi Sakai for funding his visit to the conference and making the trip to Japan possible.

\bibliographystyle{plain}
\bibliography{Logicpaperrefs}

\end{document}